\documentclass[a4paper,reqno]{amsart}

\usepackage{lmodern}
\usepackage[T1]{fontenc}
\usepackage[utf8]{inputenc}
\usepackage[english]{babel}
\usepackage{microtype} 

\usepackage{amsmath,amssymb,amsthm,mathrsfs,latexsym,mathtools,mathdots,enumerate,tikz}
\usetikzlibrary{arrows}

\usepackage[all]{xy}
\usepackage[enableskew,vcentermath]{youngtab}

\graphicspath{ {./}{./figures/} }
\ifpdf
	\usepackage{epstopdf}		
	\DeclareGraphicsExtensions{.png,.jpg,.eps,.epsf,.pdf}
\fi

\newtheorem{theorem}{Theorem}
\newtheorem{proposition}[theorem]{Proposition}
\newtheorem{lemma}[theorem]{Lemma}
\newtheorem{claim}[theorem]{Claim}
\newtheorem{definition}[theorem]{Definition}
\newtheorem{corollary}[theorem]{Corollary}
\newtheorem{example}[theorem]{Example}
\newtheorem{remark}[theorem]{Remark}
\newtheorem{conjecture}[theorem]{Conjecture}

\newcommand{\defin}[1]{{\emph{#1}}}

\newcommand{\GT}{\mathbf{GT}}

\newcommand{\xvec}{\mathbf{x}}
\newcommand{\vvec}{\mathbf{v}}

\DeclareMathOperator*{\lex}{lex}

\newcommand{\lambdavec}{{\lambda}}
\newcommand{\muvec}{{\mu}}
\newcommand{\nuvec}{{\nu}}

\newcommand{\dominate}{\trianglerighteq}
\newcommand{\lexless}{<_{\lex}}
\newcommand{\lexleq}{\leq_{\lex}}

\newcommand{\schurS}{S} 

\DeclareMathOperator*{\sign}{sgn}

\begin{document}
\title{Combinatorial proof of the skew K-saturation theorem}

\author[P.~Alexandersson]{Per Alexandersson}

\address{Institut für Mathematik,
Universität Zürich,
Winterthurerstrasse 190,
CH-8057 Zürich}
\email{per.alexandersson@math.uzh.ch, per.w.alexandersson@gmail.com}


\begin{abstract}

We give a combinatorial proof of the skew version of the K-saturation theorem.
More precisely, for any positive integer $k$,
we give an explicit injection from the set of skew semistandard Young tableaux
with skew shape $k\lambdavec/k\muvec$ and type $k\nuvec$ to the set of skew semistandard Young tableaux
of shape $\lambdavec/\muvec$ and type $\nuvec$.

Based on this method, we pose some natural conjectural refinements on related problems.
\end{abstract}

\maketitle

\section{Introduction}
We will assume that the reader is familiar with the notions of semistandard
Young tableaux and related concepts. For a good reference in this field,
see for example \cite{Stanley2001,Macdonald79symmetric}.
\medskip

The numbers $K_{\lambdavec/\muvec,\nuvec}$ are called the skew Kostka numbers
(or coefficients), and counts the number of semistandard Young tableaux with
skew shape $\lambdavec/\muvec$ and type $\nuvec$. These numbers appear in
various branches of mathematics, for example in representation theory, in the
study of symmetric functions and in algebraic geometry. The skew Kostka
coefficients expresses the skew Schur functions in the basis of the
monomial symmetric polynomials,
\[
 \schurS_{\lambdavec/\muvec}(\xvec) = \sum_{\nuvec \text{ is a partition of } |\lambdavec|-|\muvec|} K_{\lambdavec/\muvec,\nuvec} m_\nuvec.
\]
where $|\lambdavec|$ denotes the sum of all parts in the partition.
Here, $\xvec^{\nuvec}$ denotes $x_1^{\nu_1}\dotsm x_n^{\nu_n}$.
\medskip 

The main result of this paper is a combinatorial proof of the following theorem:
\begin{theorem}\label{thm:mainthm}
For fixed partitions $\lambdavec$, $\muvec$ and $\nuvec,$ the following holds
for any integer $k \geq 1\colon$
\begin{align}\label{eqn:ksaturation}
K_{k\lambdavec/k\muvec,k\nuvec} >0 \Leftrightarrow K_{\lambdavec/\muvec,\nuvec} >0.
\end{align}
\end{theorem}
Here, multiplication of a partition by a scalar is done entrywise and this is
commonly referred to as \emph{stretching}.
\medskip

An indirect proof of this was found by A.~Knutson and T.~Tao in 1998, see \cite{Knutson99thehoneycomb,Buch2000}.
They give a proof of a stronger result, namely
\begin{theorem}\label{thm:csaturation}
For fixed partitions $\lambdavec,\muvec$ and $\nuvec,$ the following holds for any integer $k \geq 1\colon$
\begin{align}\label{eqn:csaturation}
c_{k\lambdavec/k\muvec,k\nuvec} >0 \Leftrightarrow c_{\lambdavec/\muvec,\nuvec} >0
\end{align}
where the $c_{\lambdavec/\muvec,\nuvec}$ denotes the Littlewood-Richardson coefficients.
\end{theorem}
There is a parametrization of the Littewood-Richardson coefficients which gives
the skew Kostka coefficients. Using this, \eqref{eqn:csaturation} therefore
implies \eqref{eqn:ksaturation}.
The explicit conversion from skew Kostka coefficients to Littewood-Richardson coefficients
can be obtained from \cite[p. 338]{Stanley2001}.
In short, $K_{\lambdavec/\muvec,\nuvec} = \langle h_\nu, \schurS_{\lambdavec/\muvec} \rangle$,
and the latter can be expressed as a Littewood-Richardson coefficient,
using the calculations in \cite[p. 339]{Stanley2001}.
Note that if $K_{\lambdavec/\muvec,\nuvec} $ is expressed as $c_{\lambdavec'/\muvec',\nuvec'}$ for some $\lambdavec'$, 
$\muvec'$, $\nuvec'$ then  $K_{k\lambdavec/k\muvec,k\nuvec}  = c_{k\lambdavec'/k\muvec',k\nuvec'}$ for all natural numbers $k$,
that is, \emph{the correspondence commutes with stretching}. 
This shows that Theorem \ref{thm:mainthm} follows from Theorem \ref{thm:csaturation}.

\medskip

A corollary of \eqref{eqn:ksaturation} is the following statement,
(known as Fulton's K-saturation theorem, see \cite{King04stretched}),
where only non-skew tableaux are considered:
\begin{align}\label{eqn:k2saturation}
K_{k\lambdavec,k\nuvec} >0 \Leftrightarrow K_{\lambdavec,\nuvec} >0 \text{ for each integer } k\geq 1.
\end{align}
This statement can be proved directly by using the fact that
$K_{\lambdavec,\nuvec} >0$ if and only of $\lambdavec \dominate \nuvec$,
together with the fact that $k\lambdavec \dominate k\nuvec$ if and only if $\lambdavec \dominate \nuvec$.
Here, $\dominate$ denotes the \emph{domination order}.

Theorem \eqref{eqn:ksaturation} therefore lies somewhere in
between \eqref{eqn:csaturation} and \eqref{eqn:k2saturation} in difficulty.
The fact that there is no simple method to determine whenever $K_{\lambdavec/\muvec,\nuvec}$ is non-zero
strongly indicates that proving \eqref{eqn:ksaturation} is non-trivial.

The original proof of \eqref{eqn:csaturation} is rather intricate,
and more recent proofs, \cite{Buch2000} still require some technical arguments
(non-trivial bijection to hives and minimization arguments).
This motivates a more direct proof of Theorem \ref{thm:mainthm} which is presented here.

We actually prove a slightly stronger statement. The numbers $K_{k\lambdavec/k\muvec,k\nuvec}$
can be interpreted as counting certain combinatorial objects called \defin{Gelfand-Tsetlin patterns},
which are in bijection with skew semistandard Young tableaux.
We prove that that the lexicographically largest (and smallest) such Gelfand-Tsetlin pattern
is in a natural correspondence with a Gelfand-Tseltin pattern counted by $K_{\lambdavec/\muvec,\nuvec}$.

Another motivation of this paper is that the methods used may be used to give
some insight in some conjectures related to certain polynomials obtained from stretching.
It has recently been showed in \cite{Derksen2002,Rassart2004}
that $k \mapsto c_{k\lambdavec/k\muvec,k\nuvec}$ is polynomial for all fixed
$\lambdavec/\muvec$ and $\nuvec$, which therefore implies polynomiality of
stretched skew Kostka coefficients.

In \cite{King04stretched}, King, Tollu and Toumazet gave the following conjecture:
\begin{conjecture}
For any fixed $\lambdavec/\muvec$ and $\nuvec$,
the polynomial $P(k) = c_{k\lambdavec/k\muvec,k\nuvec}$ has non-negative coefficients.
\end{conjecture}
A special case of this conjecture the following:
\begin{conjecture}
For any fixed $\lambdavec/\muvec$ and $\nuvec$,
the polynomial $P(k) = K_{k\lambdavec/k\muvec,k\nuvec}$ has non-negative coefficients.
\end{conjecture}
As a consequence of this paper, we get a natural
refinement of the latter conjecture which might give some
insight in the structure of these polynomials.

\section{Proofs}

We prove \eqref{eqn:ksaturation} by working with skew Gelfand-Tsetlin patterns (GT-patterns).
GT-patterns were introduced in \cite{GelfandTsetlin50} and there is a simple bijection
between GT-patterns and semistandard Young tableaux.
Since we work with skew tableaux, we need a slight generalization of GT-patterns
introduced in \cite{Louck2003}, called \defin{skew GT-patterns}.
These objects are in bijection with skew semistandard Young tableaux
and are for our purposes easier to work with than the tableaux themselves.
There is also a bijection to non-intersecting lattice paths,
where an analogous method to our proof can be carried out.

\subsection{A bijection between skew Young tableaux and skew GT-patterns}

A skew Gelfand-Tsetlin pattern consists of non-negative integers
$(x^i_j)_{1\leq i \leq m, 1\leq j \leq n}$ arranged in a parallelogram,
satisfying the conditions $x^{i+1}_{j} \geq x^i_{j}$ and $x^i_{j} \geq x^{i+1}_{j+1}$
for all values of $i$, $j$ where the indexing is defined, as in \eqref{eq:gtpattern}.
The inequality conditions simply states that horizontal rows, down-right diagonals and up-right diagonals are weakly decreasing. 
Hence, each row $\xvec^i$ can be seen as a partition.
\begin{align}\label{eq:gtpattern}
\begin{array}{cccccccccccccc}
x^m_{1} & & x^m_{2} & & \cdots & & \cdots & & x^m_{n} \\
 & \ddots & & \ddots &  & &   & & & \ddots  \\
  &  &   x^2_{1} &  & x^2_{2} & & \cdots & & \cdots & & x^2_{n} \\
   &  &    & x^1_{1} &    & x^1_{2} & & \cdots & & \cdots & & x^1_{n}
\end{array}
\end{align}
Every skew GT-pattern with $m$ rows, top row $\lambdavec$ and bottom row $\muvec$,
define a unique skew semi-standard Young tableau of shape $\lambdavec/\muvec$
and entries in $1,2,\dots,m-1$ in the following way:

Row $\xvec^{i+1}$ and $\xvec^i$ in the GT-pattern can be viewed as two partitions which
defines the skew shape $\xvec^{i+1}/\xvec^i$.
This skew shape indicates which boxes in the corresponding
tableau that have content $i$, for $i=1,2,\dots,m-1$.
We illustrate this by an example, a formal proof of the bijection can be found in \cite{Louck2003}.

\begin{example}
The GT-pattern in \eqref{eq:bijection} corresponds to the skew Young tableau in \eqref{eq:bijection}.
For example, row $5$ and $4$ define a skew shape $(6,4,3,3)/(6,3,3,1)$
and this shape indicates exactly the location of the boxes with content $4$.
\begin{equation}\label{eq:bijection}
\setcounter{MaxMatrixCols}{20}
\begin{matrix}
6 &  & 4 &  & 3 &  & 3 &  &  &  &  \\
  & 6 &  & 3 &  & 3 &  & 1 &  &  &  \\
  &  & 6 &  & 3 &  & 1 &  & 1 &  &  \\
  &  &  & 4 &  & 2 &  & 1 &  & 1 &  \\
  &  &  &  & 3 &  & 2 &  & 1 &  & 0
\end{matrix}
\quad
\longleftrightarrow
\quad
\young(:::122,::24,:33,144)
\end{equation}
\end{example}

When talking about the shape and type of a GT-pattern, we refer to the
properties of the corresponding tableau.
We also let $\GT_{\lambdavec/\muvec,\nuvec}$ denote the set of skew GT-patterns
corresponding to tableaux of shape $\lambdavec/\muvec$ and type $\nuvec$.
\medskip

Remember, $\nuvec = (\nu_1,\nu_2,\dots)$ is the \emph{integer composition} where $\nu_i$
counts the number of boxes with content $i$ in the tableau.
Note that the cardinality of $\GT_{\lambdavec/\muvec,\nuvec}$
is invariant under permutations of the entries in $\nuvec$,
(this is due to the fact that Schur polynomials are indeed symmetric polynomials).
The GT-pattern in \eqref{eq:bijection} belong to $\GT_{\lambdavec/\muvec,\nuvec}$
where $\lambdavec=(6, 4 , 3 , 3)$, $\muvec = (3,2,1)$ and $\nuvec=(3,2,3,2)$.

\subsection{Properties of GT-patterns}

Here are some general facts about skew GT-patterns.
These are immediate consequences of the bijection above.
\begin{lemma}\label{lm:gtproperties}
Let $G \in \GT_{\lambdavec/\muvec,\nuvec}$. Then 
\begin{enumerate}[(a)]
\item The partition in the top row of $G$ is equal to $\lambdavec$.
\item The partition in the bottom row of $G$ is equal to $\muvec$.
\item If $|\xvec^i|$ denotes the sum of the entries in row $i$ in $G$, then $|\xvec^{i+1}|-|\xvec^i| = \nu_i$
for $i=1,2,\dots,m-1$ were $m$ is the number of rows of $G$.
\end{enumerate}
\end{lemma}

Using Lemma \ref{lm:gtproperties}, it follows that entrywise multiplication by $k$
of a GT-pattern $G \in \GT_{\lambdavec/\muvec,\nuvec}$
gives a GT-patten $kG$ in $\GT_{k\lambdavec/k\muvec,k\nuvec}$.
This implies the easy direction in \eqref{eqn:ksaturation} of Theorem \ref{thm:mainthm}.

As an immediate consequence of Lemma \ref{lm:gtproperties}, we also have the following:
\begin{corollary}\label{cor:sumk}
If $G \in \GT_{k\lambdavec/k\muvec,k\nuvec}$ then the sum of the entries in any row of $G$ is a multiple of $k$.
\end{corollary}

\bigskip

\subsection{Tiles and snakes in GT-patterns}

We now closely examine GT-patterns in $\GT_{k\lambdavec/k\muvec,k\nuvec}$.
The goal is to define an operation on such patterns, such that repeated application
of this operation yield a pattern where all entries are multiples of $k$.

Let $x_j^i$ be an entry in a GT-pattern.
The entries \defin{adjacent to} $x_j^i$ are the entries $x_{j}^{i+1}$, $x_{j+1}^{i+1}$,
$x_{j-1}^{i-1}$ and $x_{j}^{i-1}$, that is, the four diagonal-wise closest entries to $x_j^i$.
Of course, at the boundary of the pattern, not all four adjacent entries are present.

The \defin{tiling} of a GT-pattern $G$ is a partition of entries in $G$ into \defin{tiles}.
This partition is defined as the finest partition with the property that
entries in $G$ that are equal and adjacent belong to the same tile.
Tiles that do not contain entries from the top or bottom row are called \defin{free tiles}.
The concept of tiles has been used in other places for other purposes before, see \cite{Loera04}.

Let $T$ be a free tile in $G$.
As a consequence of the inequalities that needs to be satisfied in a GT-pattern,
the following properties hold for $T$, which are easy to verify:
\begin{enumerate}
\item If two entries in a row of $G$ belong to $T$, then all intermediate entries in the row belong to $T$.
\item If $x^{i}_j$ and $ x^{i}_{j+1}$ belong to $T$, then $ x^{i+1}_{j+1}$ and $ x^{i-1}_{j}$ also belong to $T$.
\item $T$ has a unique topmost and lowest entry, that is there is a \emph{unique} $x^{i}_j$ in $T$ that maximizes (minimizes) $i$.
\item If $T$ has $r_i$ entries in row $i$, then $r_{i+1}$ is either $r_i-1$, $r_i$ or $r_i+1$.
\end{enumerate}
Note that the first two properties imply the last two. 
See Figure \ref{fig:tilingExample} for an example of a tiling. The non-shaded tiles are free tiles.
\begin{figure}[ht!]
\centering
\begin{tikzpicture}[scale=0.4]
\draw[black] (5,-1)--(6,-2)--(7,-1)--(6,0)--(5,-1)--cycle;
\draw[black] (18,-10)--(19,-9)--(18,-8)--(19,-7)--(20,-6)--(19,-5)--(20,-4)--(19,-3)--(18,-2)--(17,-1)--(16,0)--(15,-1)--(14,-2)--(15,-3)--(16,-4)--(17,-5)--(16,-6)--(15,-7)--(16,-8)--(17,-9)--(18,-10)--cycle;
\draw[black] (4,-2)--(5,-3)--(6,-2)--(5,-1)--(4,-2)--cycle;
\draw[black] (9,-3)--(8,-4)--(9,-5)--(10,-6)--(11,-7)--(12,-6)--(11,-5)--(12,-4)--(11,-3)--(12,-2)--(11,-1)--(10,-2)--(9,-3)--cycle;
\draw[black] (12,-8)--(13,-9)--(14,-8)--(15,-7)--(16,-6)--(17,-5)--(16,-4)--(15,-3)--(14,-2)--(13,-1)--(12,-2)--(11,-3)--(12,-4)--(11,-5)--(12,-6)--(13,-7)--(12,-8)--cycle;
\draw[black] (5,-3)--(6,-4)--(7,-3)--(6,-2)--(5,-3)--cycle;
\draw[black] (6,-4)--(7,-5)--(8,-4)--(7,-3)--(6,-4)--cycle;
\draw[black] (11,-7)--(10,-6)--(9,-5)--(8,-4)--(7,-5)--(8,-6)--(9,-7)--(10,-8)--(11,-7)--cycle;
\draw[black] (20,-8)--(21,-7)--(22,-6)--(21,-5)--(20,-4)--(19,-5)--(20,-6)--(19,-7)--(18,-8)--(19,-9)--(20,-8)--cycle;
\draw[black] (11,-7)--(12,-8)--(13,-7)--(12,-6)--(11,-7)--cycle;
\draw[black] (12,-8)--(11,-7)--(10,-8)--(11,-9)--(12,-10)--(13,-9)--(12,-8)--cycle;
\filldraw[color=black,fill=lightgray] (12,-10)--(13,-11)--(14,-10)--(13,-9)--(12,-10)--cycle;
\filldraw[color=black,fill=lightgray] (12,0)--(11,1)--(10,0)--(11,-1)--(12,-2)--(13,-1)--(12,0)--cycle;
\filldraw[color=black,fill=lightgray] (6,0)--(5,1)--(4,0)--(3,1)--(2,0)--(3,-1)--(4,-2)--(5,-1)--(6,0)--cycle;
\filldraw[color=black,fill=lightgray] (16,0)--(15,1)--(14,0)--(13,1)--(12,0)--(13,-1)--(14,-2)--(15,-1)--(16,0)--cycle;
\filldraw[color=black,fill=lightgray] (14,-8)--(13,-9)--(14,-10)--(15,-11)--(16,-10)--(17,-11)--(18,-10)--(17,-9)--(16,-8)--(15,-7)--(14,-8)--cycle;
\filldraw[color=black,fill=lightgray] (9,-3)--(10,-2)--(11,-1)--(10,0)--(9,1)--(8,0)--(7,1)--(6,0)--(7,-1)--(6,-2)--(7,-3)--(8,-4)--(9,-3)--cycle;
\filldraw[color=black,fill=lightgray] (21,-7)--(20,-8)--(19,-9)--(18,-10)--(19,-11)--(20,-10)--(21,-11)--(22,-10)--(23,-11)--(24,-10)--(25,-11)--(26,-10)--(25,-9)--(24,-8)--(23,-7)--(22,-6)--(21,-7)--cycle;
\node at (3,0) {$15$};
\node at (5,0) {$15$};
\node at (7,0) {$9$};
\node at (9,0) {$9$};
\node at (11,0) {$6$};
\node at (13,0) {$3$};
\node at (15,0) {$3$};
\node at (4,-1) {$15$};
\node at (6,-1) {$10$};
\node at (8,-1) {$9$};
\node at (10,-1) {$9$};
\node at (12,-1) {$6$};
\node at (14,-1) {$3$};
\node at (16,-1) {$2$};
\node at (5,-2) {$14$};
\node at (7,-2) {$9$};
\node at (9,-2) {$9$};
\node at (11,-2) {$7$};
\node at (13,-2) {$5$};
\node at (15,-2) {$2$};
\node at (17,-2) {$2$};
\node at (6,-3) {$12$};
\node at (8,-3) {$9$};
\node at (10,-3) {$7$};
\node at (12,-3) {$5$};
\node at (14,-3) {$5$};
\node at (16,-3) {$2$};
\node at (18,-3) {$2$};
\node at (7,-4) {$11$};
\node at (9,-4) {$7$};
\node at (11,-4) {$7$};
\node at (13,-4) {$5$};
\node at (15,-4) {$5$};
\node at (17,-4) {$2$};
\node at (19,-4) {$2$};
\node at (8,-5) {$8$};
\node at (10,-5) {$7$};
\node at (12,-5) {$5$};
\node at (14,-5) {$5$};
\node at (16,-5) {$5$};
\node at (18,-5) {$2$};
\node at (20,-5) {$1$};
\node at (9,-6) {$8$};
\node at (11,-6) {$7$};
\node at (13,-6) {$5$};
\node at (15,-6) {$5$};
\node at (17,-6) {$2$};
\node at (19,-6) {$2$};
\node at (21,-6) {$1$};
\node at (10,-7) {$8$};
\node at (12,-7) {$6$};
\node at (14,-7) {$5$};
\node at (16,-7) {$2$};
\node at (18,-7) {$2$};
\node at (20,-7) {$1$};
\node at (22,-7) {$0$};
\node at (11,-8) {$7$};
\node at (13,-8) {$5$};
\node at (15,-8) {$3$};
\node at (17,-8) {$2$};
\node at (19,-8) {$1$};
\node at (21,-8) {$0$};
\node at (23,-8) {$0$};
\node at (12,-9) {$7$};
\node at (14,-9) {$3$};
\node at (16,-9) {$3$};
\node at (18,-9) {$2$};
\node at (20,-9) {$0$};
\node at (22,-9) {$0$};
\node at (24,-9) {$0$};
\node at (13,-10) {$6$};
\node at (15,-10) {$3$};
\node at (17,-10) {$3$};
\node at (19,-10) {$0$};
\node at (21,-10) {$0$};
\node at (23,-10) {$0$};
\node at (25,-10) {$0$};
\end{tikzpicture}
\caption{The partitioning of entries of a GT-pattern into tiles.}\label{fig:tilingExample}
\end{figure}
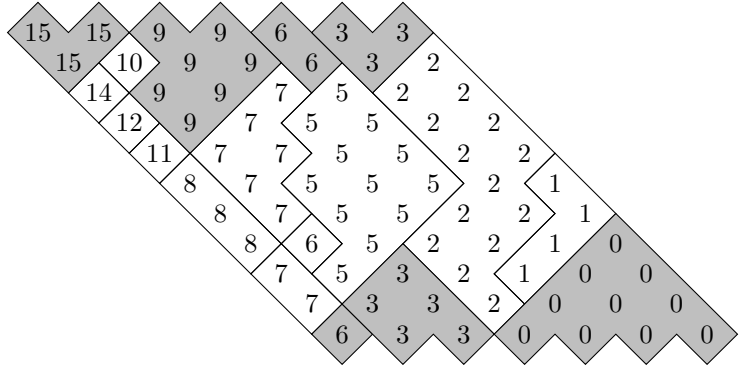
A \defin{snake} is any subset of a tile in $G$ that has at most one entry from each row in $G$ and the entries in this set constitute
a connected component under the adjacency relations defined above.

If $x^i_j$ is the entry in a snake $S$ with maximal $i$, we say that $S$ starts at row $i$.
Similarly, if $x^i_j$ is the entry in a snake $S$ with minimal $i$
then we say that $S$ ends at row $i-1$.

\medskip 

A free tile $T$ partitioned into disjoint snakes, is called a \defin{snake partition} of $T$.
A snake partition is \defin{proper} if at most one snake start or end at each row of $T$.
A proper snake partition of a free tile satisfies the following properties:
\begin{lemma}\label{lem:snakesstartstop}
Let $T$ be a free tile with a proper snake partition and let $r_i$ denote the number of entries in $T$ in row $i$ of $G$.
Then
\begin{itemize}
\item If $r_{i+1}< r_{i}$ then exactly one snake in $T$ starts in row $i$ in $T$.
\item If $r_{i+1}> r_{i}$ then exactly one snake in $T$ ends in row $i-1$ in $T$.
\item If $r_{i+1} = r_{i}$ then row $i$ and $i+1$ in $T$ are identical, that is,
if entry $j$ in row $i$ and entry $j$ in row $i+1$ belong to the same snake.
\end{itemize}
\end{lemma}
\begin{proof}
Case 1: The $r_{i+1}$ entries in row $i+1$ belong to different snakes,
and $r_{i} = r_{i+1}+1$ entries in row $i$ belong to different snakes.
Therefore, at least one snake must start in row $i$. If more than one snake
start in this row, there is not enough space for the $r_{i+1}$ snakes in the 
row above to be connected to the $r_i$ entries in row $i$.
This would force at least one snake from row $i+1$ to end in row $i$,
which prevents the snake partition from being proper.

The other two cases are treated in a similar manner.
\end{proof}

\begin{definition}\label{def:snakes}
Let $G\in\GT_{k\lambdavec/k\muvec,k\nuvec}$ and let $T$
be a free tile of $G$ where the entries of $T$ are not divisible by $k$.

Let $R_i$ be the $i$th row of $T$, where $R_1$ has only 
one entry corresponding to the unique lowest entry in $T$,
and let $R_{ij}$ be the $j$th entry in row $i$ of $T$.

For each $i$, partition the set of entries $R_{i1},R_{i2},\dots$ into connected components.
\emph{Note, for some values of $i$ and $j$, there is no entry $R_{ij}$.}
By construction, the connected components have at most one entry from each row,
so this is a partition of the entries in $T$ into disjoint snakes.
This is the \defin{canonical snake partition} of the tile $T$.
\end{definition}
It is clear from this definition that adjacent rows in the canonical snake partition have
one of the following forms, corresponding to the cases in Lemma \ref{lem:snakesstartstop}.
Here, entries are enumerated according to which snake they are members of.
Underlined entries denote topmost or lowest entry of a snake.

\noindent
\textbf{Case 1:}
\begin{equation}
\begin{matrix}
& 1 && 2 && \dots  && l \\
1 && 2 && \dots && l && \underline{l+1} \\
\end{matrix}
\end{equation}
\textbf{Case 2:}
\begin{equation}
\begin{matrix}
1 && 2 && \dots && l && \underline{l+1} \\
& 1 && 2 && \dots  && l \\
\end{matrix}
\end{equation}
\textbf{Case 3:}
\begin{equation}
\begin{matrix}
1 && 2 && \dots  && l \\
& 1 && 2 && \dots  && l \\
\end{matrix}
\qquad
\text{  or  }
\qquad
\begin{matrix}
& 1 && 2 && \dots  && l \\
1 && 2 && \dots  && l \\
\end{matrix}
\end{equation}
From this observation, it is clear that the canonical snake partition
of a free tile is proper.

Partitioning all free tiles in $G$ with entries not divisible by $k$,
gives the \defin{canonical snake partition of $G$}.

Note that an entry in $G$ which is not a multiple of $k$ is a member of some (unique) snake
in the canonical snake partition of $G$.

\medskip

\begin{example}\label{ex:snakes}
An element from $\GT_{k\lambdavec/k\muvec,k\nuvec}$ for $k=3$,
$\lambdavec=(5,5,3,3,2,1,1)$, $\muvec=(2,1,1)$
and $\nuvec=(1, 1, 2, 2, 1, 2, 1, 2, 2, 2)$ is displayed in Figure \ref{fig:snakesExample}.
Snake number $5$ starts at row 9 and ends at row 2.
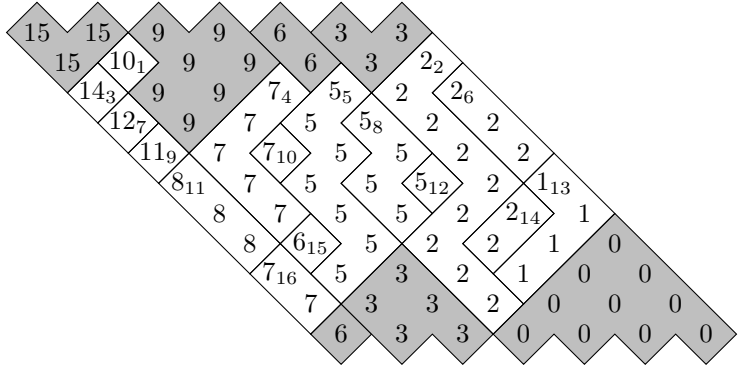
\begin{figure}[ht!]
\centering
\begin{tikzpicture}[scale=0.4]
\filldraw[color=black,fill=lightgray] (12,-10)--(13,-11)--(14,-10)--(13,-9)--(12,-10)--cycle;
\filldraw[color=black,fill=lightgray] (12,0)--(11,1)--(10,0)--(11,-1)--(12,-2)--(13,-1)--(12,0)--cycle;
\filldraw[color=black,fill=lightgray] (6,0)--(5,1)--(4,0)--(3,1)--(2,0)--(3,-1)--(4,-2)--(5,-1)--(6,0)--cycle;
\filldraw[color=black,fill=lightgray] (16,0)--(15,1)--(14,0)--(13,1)--(12,0)--(13,-1)--(14,-2)--(15,-1)--(16,0)--cycle;
\filldraw[color=black,fill=lightgray] (14,-8)--(13,-9)--(14,-10)--(15,-11)--(16,-10)--(17,-11)--(18,-10)--(17,-9)--(16,-8)--(15,-7)--(14,-8)--cycle;
\filldraw[color=black,fill=lightgray] (9,-3)--(10,-2)--(11,-1)--(10,0)--(9,1)--(8,0)--(7,1)--(6,0)--(7,-1)--(6,-2)--(7,-3)--(8,-4)--(9,-3)--cycle;
\filldraw[color=black,fill=lightgray] (21,-7)--(20,-8)--(19,-9)--(18,-10)--(19,-11)--(20,-10)--(21,-11)--(22,-10)--(23,-11)--(24,-10)--(25,-11)--(26,-10)--(25,-9)--(24,-8)--(23,-7)--(22,-6)--(21,-7)--cycle;
%
\draw[black] (5,-1)--(6,-2)--(7,-1)--(6,0)--(5,-1)--cycle;
\draw[black] (18,-4)--(17,-3)--(16,-2)--(17,-1)--(16,0)--(15,-1)--(14,-2)--(15,-3)--(16,-4)--(17,-5)--(16,-6)--(15,-7)--(16,-8)--(17,-9)--(18,-10)--(19,-9)--(18,-8)--(17,-7)--(18,-6)--(19,-5)--(18,-4)--cycle;
\draw[black] (4,-2)--(5,-3)--(6,-2)--(5,-1)--(4,-2)--cycle;
\draw[black] (11,-1)--(10,-2)--(9,-3)--(8,-4)--(9,-5)--(10,-6)--(11,-7)--(12,-6)--(11,-5)--(10,-4)--(11,-3)--(12,-2)--(11,-1)--cycle;
\draw[black] (11,-3)--(12,-4)--(11,-5)--(12,-6)--(13,-7)--(12,-8)--(13,-9)--(14,-8)--(15,-7)--(14,-6)--(13,-5)--(14,-4)--(13,-3)--(14,-2)--(13,-1)--(12,-2)--(11,-3)--cycle;
\draw[black] (20,-4)--(19,-3)--(18,-2)--(17,-1)--(16,-2)--(17,-3)--(18,-4)--(19,-5)--(20,-4)--cycle;
\draw[black] (5,-3)--(6,-4)--(7,-3)--(6,-2)--(5,-3)--cycle;
\draw[black] (15,-7)--(16,-6)--(15,-5)--(16,-4)--(15,-3)--(14,-2)--(13,-3)--(14,-4)--(13,-5)--(14,-6)--(15,-7)--cycle;
\draw[black] (6,-4)--(7,-5)--(8,-4)--(7,-3)--(6,-4)--cycle;
\draw[black] (10,-4)--(11,-5)--(12,-4)--(11,-3)--(10,-4)--cycle;
\draw[black] (11,-7)--(10,-6)--(9,-5)--(8,-4)--(7,-5)--(8,-6)--(9,-7)--(10,-8)--(11,-7)--cycle;
\draw[black] (15,-5)--(16,-6)--(17,-5)--(16,-4)--(15,-5)--cycle;
\draw[black] (20,-8)--(21,-7)--(22,-6)--(21,-5)--(20,-4)--(19,-5)--(20,-6)--(19,-7)--(18,-8)--(19,-9)--(20,-8)--cycle;
\draw[black] (17,-7)--(18,-8)--(19,-7)--(20,-6)--(19,-5)--(18,-6)--(17,-7)--cycle;
\draw[black] (11,-7)--(12,-8)--(13,-7)--(12,-6)--(11,-7)--cycle;
\draw[black] (12,-8)--(11,-7)--(10,-8)--(11,-9)--(12,-10)--(13,-9)--(12,-8)--cycle;
\node at (3,0) {$15$};
\node at (5,0) {$15$};
\node at (7,0) {$9$};
\node at (9,0) {$9$};
\node at (11,0) {$6$};
\node at (13,0) {$3$};
\node at (15,0) {$3$};
\node at (4,-1) {$15$};
\node at (6,-1) {$10_1$};
\node at (8,-1) {$9$};
\node at (10,-1) {$9$};
\node at (12,-1) {$6$};
\node at (14,-1) {$3$};
\node at (16,-1) {$2_2$};
\node at (5,-2) {$14_3$};
\node at (7,-2) {$9$};
\node at (9,-2) {$9$};
\node at (11,-2) {$7_4$};
\node at (13,-2) {$5_5$};
\node at (15,-2) {$2$};
\node at (17,-2) {$2_6$};
\node at (6,-3) {$12_7$};
\node at (8,-3) {$9$};
\node at (10,-3) {$7$};
\node at (12,-3) {$5$};
\node at (14,-3) {$5_8$};
\node at (16,-3) {$2$};
\node at (18,-3) {$2$};
\node at (7,-4) {$11_9$};
\node at (9,-4) {$7$};
\node at (11,-4) {$7_{10}$};
\node at (13,-4) {$5$};
\node at (15,-4) {$5$};
\node at (17,-4) {$2$};
\node at (19,-4) {$2$};
\node at (8,-5) {$8_{11}$};
\node at (10,-5) {$7$};
\node at (12,-5) {$5$};
\node at (14,-5) {$5$};
\node at (16,-5) {$5_{12}$};
\node at (18,-5) {$2$};
\node at (20,-5) {$1_{13}$};
\node at (9,-6) {$8$};
\node at (11,-6) {$7$};
\node at (13,-6) {$5$};
\node at (15,-6) {$5$};
\node at (17,-6) {$2$};
\node at (19,-6) {$2_{14}$};
\node at (21,-6) {$1$};
\node at (10,-7) {$8$};
\node at (12,-7) {$6_{15}$};
\node at (14,-7) {$5$};
\node at (16,-7) {$2$};
\node at (18,-7) {$2$};
\node at (20,-7) {$1$};
\node at (22,-7) {$0$};
\node at (11,-8) {$7_{16}$};
\node at (13,-8) {$5$};
\node at (15,-8) {$3$};
\node at (17,-8) {$2$};
\node at (19,-8) {$1$};
\node at (21,-8) {$0$};
\node at (23,-8) {$0$};
\node at (12,-9) {$7$};
\node at (14,-9) {$3$};
\node at (16,-9) {$3$};
\node at (18,-9) {$2$};
\node at (20,-9) {$0$};
\node at (22,-9) {$0$};
\node at (24,-9) {$0$};
\node at (13,-10) {$6$};
\node at (15,-10) {$3$};
\node at (17,-10) {$3$};
\node at (19,-10) {$0$};
\node at (21,-10) {$0$};
\node at (23,-10) {$0$};
\node at (25,-10) {$0$};
\end{tikzpicture}
\caption{A GT-pattern with the snake partition using Definition \ref{def:snakes}.}\label{fig:snakesExample}
\end{figure}
\end{example}

Notice that in Example \ref{ex:snakes}, there is no row with a single snake entry.
This is true in general: For any row $\xvec$ in $G \in \GT_{k\lambdavec/k\muvec,k\nuvec}$,
the sum of the entries is a multiple of $k$.
Non-snake entries are multiples of $k$, so the snake entries in a row must also sum to a multiple of $k$.
This simple observation implies the following:
\begin{corollary}\label{corr:snakedegree}
If a snake in $G$ starts (ends) at some row $j$, at least one other snake in $G$ starts or ends
at the same row.
\end{corollary}

\medskip

\subsection{Snake movements}\label{sec:snakemoves}

Let $T$ be a free tile with a proper snake partition.
A snake $S_1$ is \defin{to the left (right) of} $S_2$ if there is some row of $T$
with an entry of $S_1$ to the left (right) of $S_2$.
A snake $S_1$ is \defin{immediately to the left (right) of} $S_2$ if there is no snake $S_3$
such that $S_3$ is to the right (left) of $S_1$ but to the left (right) of $S_2$.

\begin{lemma}\label{lem:snakemoves}
Let $S_1$ and $S_2$ be two snakes in a proper snake partition such that $S_1$ is immediately to the left (right) of $S_2$.
Consider the action where the snake membership is interchanged between these snakes, for all entries in rows that
contain entries from both snakes. Then the result of this action is still two (connected) snakes,
and the resulting snake partition is still proper.
\end{lemma}
\begin{proof}
The snakes must start (and end) in different rows, since these are snakes from a proper snake partition.
Assume that $S_1$ is to the left of $S_2$ and starts higher up than $S_2$. Consider the row where $S_2$ starts,
where an underlined entry still denotes a topmost or bottommost entry in a snake:
\begin{equation}
\begin{matrix}
& \star \\
1 && \underline{2} \\
\vdots && \vdots \\
\end{matrix}
\end{equation}
The entry marked with $\star$ must be a member of $S_1$, since otherwise, 
some third snake $S_3$ must have its lowest entry there.
But then, $S_3$ ends in the same row as $S_2$ starts, which violates Lemma \ref{lem:snakesstartstop}.

Using a similar reasoning for the other cases, we conclude that for two snakes that are immediately adjacent,
they are arranged in one of the four configurations below:
\begin{equation}\label{eq:snakeconfigs}
\begin{matrix}
& \vdots \\
& 1 \\
1 && \underline{2} \\
\vdots && \vdots \\
1 && \underline{2} \\
& 1 \\
& \vdots
\end{matrix}
\qquad \text{ or } \qquad
\begin{matrix}
& \vdots \\
& 1 \\
\underline{2} && 1 \\
\vdots && \vdots \\
\underline{2} && 1 \\
& 1 \\
& \vdots
\end{matrix}
\qquad \text{ or } \qquad
\begin{matrix}
& \vdots \\
& 1 \\
1 && \underline{2} \\
\vdots && \vdots \\
\underline{1} && 2 \\
& 2 \\
& \vdots
\end{matrix}
\qquad \text{ or } \qquad
\begin{matrix}
& \vdots \\
& 1 \\
\underline{2} && 1 \\
\vdots && \vdots \\
2 && \underline{1} \\
& 2 \\
& \vdots
\end{matrix}
\end{equation}
The action of switching the snake membership of entries in rows that contains entries from both snakes,
is just the action of interchanging configuration 1 and 2, or 3 and 4.
It is evident that this action preserves connectedness of the snakes involved,
and that the resulting snake partition is still proper.
\end{proof}
\medskip

Figure \ref{fig:snakemoveexamples} shows two examples where the 
actions in Lemma \ref{lem:snakemoves} are demonstrated.
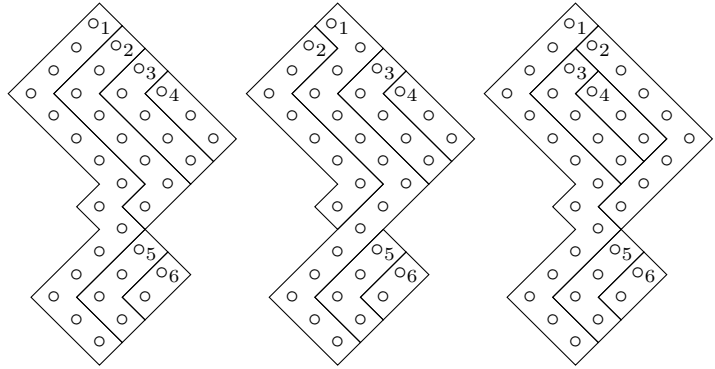
\begin{figure}[ht!]
\centering
\begin{tikzpicture}[scale=0.3]
\draw[black] (13,-3)--(14,-4)--(15,-5)--(16,-6)--(17,-7)--(16,-8)--(17,-9)--(16,-10)--(15,-11)--(14,-12)--(15,-13)--(16,-14)--(17,-15)--(18,-14)--(17,-13)--(16,-12)--(17,-11)--(18,-10)--(19,-9)--(18,-8)--(19,-7)--(18,-6)--(17,-5)--(16,-4)--(15,-3)--(16,-2)--(17,-1)--(18,0)--(17,1)--(16,0)--(15,-1)--(14,-2)--(13,-3)--cycle;
\draw[black] (17,-5)--(18,-6)--(19,-7)--(18,-8)--(19,-9)--(20,-8)--(21,-7)--(20,-6)--(19,-5)--(18,-4)--(17,-3)--(18,-2)--(19,-1)--(18,0)--(17,-1)--(16,-2)--(15,-3)--(16,-4)--(17,-5)--cycle;
\draw[black] (19,-5)--(20,-6)--(21,-7)--(22,-6)--(21,-5)--(20,-4)--(19,-3)--(20,-2)--(19,-1)--(18,-2)--(17,-3)--(18,-4)--(19,-5)--cycle;
\draw[black] (23,-5)--(22,-4)--(21,-3)--(20,-2)--(19,-3)--(20,-4)--(21,-5)--(22,-6)--(23,-5)--cycle;
\draw[black] (18,-14)--(19,-13)--(18,-12)--(19,-11)--(20,-10)--(19,-9)--(18,-10)--(17,-11)--(16,-12)--(17,-13)--(18,-14)--cycle;
\draw[black] (18,-12)--(19,-13)--(20,-12)--(21,-11)--(20,-10)--(19,-11)--(18,-12)--cycle;
\node at (17,0) {$\circ_1$};
\node at (16,-1) {$\circ$};
\node at (18,-1) {$\circ_2$};
\node at (15,-2) {$\circ$};
\node at (17,-2) {$\circ$};
\node at (19,-2) {$\circ_3$};
\node at (14,-3) {$\circ$};
\node at (16,-3) {$\circ$};
\node at (18,-3) {$\circ$};
\node at (20,-3) {$\circ_4$};
\node at (15,-4) {$\circ$};
\node at (17,-4) {$\circ$};
\node at (19,-4) {$\circ$};
\node at (21,-4) {$\circ$};
\node at (16,-5) {$\circ$};
\node at (18,-5) {$\circ$};
\node at (20,-5) {$\circ$};
\node at (22,-5) {$\circ$};
\node at (17,-6) {$\circ$};
\node at (19,-6) {$\circ$};
\node at (21,-6) {$\circ$};
\node at (18,-7) {$\circ$};
\node at (20,-7) {$\circ$};
\node at (17,-8) {$\circ$};
\node at (19,-8) {$\circ$};
\node at (18,-9) {$\circ$};
\node at (17,-10) {$\circ$};
\node at (19,-10) {$\circ_5$};
\node at (16,-11) {$\circ$};
\node at (18,-11) {$\circ$};
\node at (20,-11) {$\circ_6$};
\node at (15,-12) {$\circ$};
\node at (17,-12) {$\circ$};
\node at (19,-12) {$\circ$};
\node at (16,-13) {$\circ$};
\node at (18,-13) {$\circ$};
\node at (17,-14) {$\circ$};
\end{tikzpicture}
\begin{tikzpicture}[scale=0.3]
\draw[black] (16,-14)--(17,-15)--(18,-14)--(17,-13)--(16,-12)--(17,-11)--(18,-10)--(19,-9)--(20,-8)--(21,-7)--(20,-6)--(19,-5)--(18,-4)--(17,-3)--(18,-2)--(19,-1)--(18,0)--(17,1)--(16,0)--(17,-1)--(16,-2)--(15,-3)--(16,-4)--(17,-5)--(18,-6)--(19,-7)--(18,-8)--(17,-9)--(16,-10)--(15,-11)--(14,-12)--(15,-13)--(16,-14)--cycle;
\draw[black] (15,-5)--(16,-6)--(17,-7)--(16,-8)--(17,-9)--(18,-8)--(19,-7)--(18,-6)--(17,-5)--(16,-4)--(15,-3)--(16,-2)--(17,-1)--(16,0)--(15,-1)--(14,-2)--(13,-3)--(14,-4)--(15,-5)--cycle;
\draw[black] (19,-5)--(20,-6)--(21,-7)--(22,-6)--(21,-5)--(20,-4)--(19,-3)--(20,-2)--(19,-1)--(18,-2)--(17,-3)--(18,-4)--(19,-5)--cycle;
\draw[black] (23,-5)--(22,-4)--(21,-3)--(20,-2)--(19,-3)--(20,-4)--(21,-5)--(22,-6)--(23,-5)--cycle;
\draw[black] (18,-14)--(19,-13)--(18,-12)--(19,-11)--(20,-10)--(19,-9)--(18,-10)--(17,-11)--(16,-12)--(17,-13)--(18,-14)--cycle;
\draw[black] (18,-12)--(19,-13)--(20,-12)--(21,-11)--(20,-10)--(19,-11)--(18,-12)--cycle;
\node at (17,0) {$\circ_1$};
\node at (16,-1) {$\circ_2$};
\node at (18,-1) {$\circ$};
\node at (15,-2) {$\circ$};
\node at (17,-2) {$\circ$};
\node at (19,-2) {$\circ_3$};
\node at (14,-3) {$\circ$};
\node at (16,-3) {$\circ$};
\node at (18,-3) {$\circ$};
\node at (20,-3) {$\circ_4$};
\node at (15,-4) {$\circ$};
\node at (17,-4) {$\circ$};
\node at (19,-4) {$\circ$};
\node at (21,-4) {$\circ$};
\node at (16,-5) {$\circ$};
\node at (18,-5) {$\circ$};
\node at (20,-5) {$\circ$};
\node at (22,-5) {$\circ$};
\node at (17,-6) {$\circ$};
\node at (19,-6) {$\circ$};
\node at (21,-6) {$\circ$};
\node at (18,-7) {$\circ$};
\node at (20,-7) {$\circ$};
\node at (17,-8) {$\circ$};
\node at (19,-8) {$\circ$};
\node at (18,-9) {$\circ$};
\node at (17,-10) {$\circ$};
\node at (19,-10) {$\circ_5$};
\node at (16,-11) {$\circ$};
\node at (18,-11) {$\circ$};
\node at (20,-11) {$\circ_6$};
\node at (15,-12) {$\circ$};
\node at (17,-12) {$\circ$};
\node at (19,-12) {$\circ$};
\node at (16,-13) {$\circ$};
\node at (18,-13) {$\circ$};
\node at (17,-14) {$\circ$};
\end{tikzpicture}
\begin{tikzpicture}[scale=0.3]
\draw[black] (13,-3)--(14,-4)--(15,-5)--(16,-6)--(17,-7)--(16,-8)--(17,-9)--(16,-10)--(15,-11)--(14,-12)--(15,-13)--(16,-14)--(17,-15)--(18,-14)--(17,-13)--(16,-12)--(17,-11)--(18,-10)--(19,-9)--(18,-8)--(19,-7)--(18,-6)--(17,-5)--(16,-4)--(15,-3)--(16,-2)--(17,-1)--(18,0)--(17,1)--(16,0)--(15,-1)--(14,-2)--(13,-3)--cycle;
\draw[black] (21,-5)--(20,-6)--(19,-7)--(18,-8)--(19,-9)--(20,-8)--(21,-7)--(22,-6)--(23,-5)--(22,-4)--(21,-3)--(20,-2)--(19,-1)--(18,0)--(17,-1)--(18,-2)--(19,-3)--(20,-4)--(21,-5)--cycle;
\draw[black] (17,-5)--(18,-6)--(19,-7)--(20,-6)--(19,-5)--(18,-4)--(17,-3)--(18,-2)--(17,-1)--(16,-2)--(15,-3)--(16,-4)--(17,-5)--cycle;
\draw[black] (21,-5)--(20,-4)--(19,-3)--(18,-2)--(17,-3)--(18,-4)--(19,-5)--(20,-6)--(21,-5)--cycle;
\draw[black] (18,-14)--(19,-13)--(18,-12)--(19,-11)--(20,-10)--(19,-9)--(18,-10)--(17,-11)--(16,-12)--(17,-13)--(18,-14)--cycle;
\draw[black] (18,-12)--(19,-13)--(20,-12)--(21,-11)--(20,-10)--(19,-11)--(18,-12)--cycle;
\node at (17,0) {$\circ_1$};
\node at (16,-1) {$\circ$};
\node at (18,-1) {$\circ_2$};
\node at (15,-2) {$\circ$};
\node at (17,-2) {$\circ_3$};
\node at (19,-2) {$\circ$};
\node at (14,-3) {$\circ$};
\node at (16,-3) {$\circ$};
\node at (18,-3) {$\circ_4$};
\node at (20,-3) {$\circ$};
\node at (15,-4) {$\circ$};
\node at (17,-4) {$\circ$};
\node at (19,-4) {$\circ$};
\node at (21,-4) {$\circ$};
\node at (16,-5) {$\circ$};
\node at (18,-5) {$\circ$};
\node at (20,-5) {$\circ$};
\node at (22,-5) {$\circ$};
\node at (17,-6) {$\circ$};
\node at (19,-6) {$\circ$};
\node at (21,-6) {$\circ$};
\node at (18,-7) {$\circ$};
\node at (20,-7) {$\circ$};
\node at (17,-8) {$\circ$};
\node at (19,-8) {$\circ$};
\node at (18,-9) {$\circ$};
\node at (17,-10) {$\circ$};
\node at (19,-10) {$\circ_5$};
\node at (16,-11) {$\circ$};
\node at (18,-11) {$\circ$};
\node at (20,-11) {$\circ_6$};
\node at (15,-12) {$\circ$};
\node at (17,-12) {$\circ$};
\node at (19,-12) {$\circ$};
\node at (16,-13) {$\circ$};
\node at (18,-13) {$\circ$};
\node at (17,-14) {$\circ$};
\end{tikzpicture}
\caption{
The leftmost figure is the canonical snake partition of a tile.
The middle figure is the result after interchanging snakes $1$ and $2$,
and the last figure is the result after interchanging 
snakes $2$ and $3$ followed by $2$ and $4$ in the first figure.
}\label{fig:snakemoveexamples}
\end{figure}

\medskip

A snake in $T$ is a \defin{leftmost (rightmost)} snake in $T$
if there is no other snake in $T$ to the left (right) of it.

\begin{lemma}\label{lem:addingisok}
Let $G$ be a GT-pattern with a free tile  $T$ with a proper snake partition.
Suppose $S$ is a leftmost (rightmost) snake in $T$.
Then, adding (subtracting) $1$ to (from) all entries in $S$ gives a GT-pattern $G'$, 
where all inequalities are satisfied, and $T\setminus S$ is a tile or a disjoin union of tiles of $G'$.
\end{lemma}
\begin{proof}
Let us first consider the case of a leftmost snake $S$ and consider an entry $x^i_j$ in $S$.
Since the partition of $T$ is proper, any entry $x^{i'}_{j'}$ which is top-left or 
bottom-left of $x^i_j$ is either an entry in $S$ or not in $T$. 

The inequalities in a GT-pattern ensures that entries outside $T$ top-left or bottom-left of entries in $S$
are strictly greater than entries in $S$. Therefore, adding 1 to all elements in $S$
does not violate the inequalities fulfilled in a GT-pattern.
Furthermore, the entries top-right and bottom-right of entries 
in $S$ are clearly smaller or equal to entries in $S$ after increasing entries in $S$.

A similar reasoning works for the case of a rightmost snake in which all elements are decreased by 1.
It is clear that $T \setminus S$ is a tile or a disjoin union of tiles of $G'$ (since no other entries except those in $S$
has been changed) and that the snakes in $T \setminus S$ constitute proper 
snake partitions of the tiles in $T \setminus S$.
\end{proof}
\medskip

The following two lemmas are used to handle a technical special
case later on.

\begin{lemma}\label{lem:adjacenttiles}
Let $T_1$ and $T_2$ be free tiles in a GT-pattern $G$ with content $c+1$ respective $c$ for some integer $c$.
Assume $T_1$ and $T_2$ are partitioned with proper snake partitions.
Then, if a snake starts (ends) at row $i$ in $T_1$ then there is no snake that starts (ends) at row $i$ in $T_2$.
\end{lemma}
\begin{proof}
Assume that there is a snake that starts at row $i$ in both $T_1$ and $T_2$.
Lemma \ref{lem:snakesstartstop} gives that if
$r_{i}$ is the number of entries in row $i$ in $T_1$, then
$r_i = r_{i+1}+1$. Similarly, if $s_{i}$ is the number of entries in row $i$ in $T_2$,
then $s_i = s_{i+1}+1$.
Hence, row $i+1$ and $i$ look like 
\[
 \begin{matrix}
& \star && \star && \dots  && \star \\
\star && \star && \dots && \star && \star \\
\end{matrix}
\]
in both $T_1$ and $T_2$.

Since each row in $G$ is decreasing, the entries in row $i$ and $i+1$ in $T_1$
and $T_2$ must be consecutive in row $i$ and $i+1$ respectively.
But it is impossible to have two arrangements as above, where the entries in respective rows are adjacent.

A similar reasoning proves that two snakes from $T_1$ respective $T_2$ cannot end at the same row.
\end{proof}

\begin{lemma}\label{lem:tileunion}
Let $T_1$ and $T_2$ be free tiles in a GT-pattern $G$ with content $c+1$ respective $c$ for some $c$,
and that there is some entry in $T_1$ adjacent to some entry in $T_2$.
Then $T_1 \cup T_2$ has the shape of a free tile, that is, these entries fulfills the four properties that
a free tile fulfills.
\end{lemma}
\begin{proof}
It suffices to prove the first two properties, since these implies the last two.
The first property, is clear, if $x^i_j \in T_1 \cup T_2$
and $x^i_l \in T_1 \cup T_2$ clearly all entries between $x^i_j$ and $x^i_l$ in row $i$
are either $c$ or $c+1$, so these must be members of $T_1$ or $T_2$.

Now, if $x^i_j = c+1$ and $x^i_{j+1} = c$ and both entries are in
$T_1 \cup T_2$, then $x^{i+1}_{j+1}$ and $x^{i-1}_{j}$ must be either $c$ or $c+1$,
(since $c+1=x^i_j \geq x^{i+1}_{j+1} \geq x^i_{j+1} = c$ and $c+1=x^i_j \geq x^{i-1}_{j} \geq x^i_{j+1} = c$),
so  $x^{i+1}_{j+1}$ and $x^{i-1}_{j}$ are also entries in $T_1 \cup T_2$. This establishes the proof.
\end{proof}

\bigskip

\subsection{Snake cycles}

Given $G \in \GT_{k\lambdavec/k\muvec,k\nuvec}$ construct an undirected, loop-free graph $\Gamma_G$
(with possible multiple edges) as follows:
The vertices of $\Gamma_G$ corresponds to rows of $G$
and every snake starting at row $i$ and ending at row $j$ give an edge $(i,j)$ in $\Gamma_G$.
Corollary \ref{corr:snakedegree} implies that the degree of 
each vertex in $\Gamma_G$ is either zero or at least two.
It is easy to show that any such graph contains at least one simple cycle
(each vertex appear at most once in the cycle), if the snake partition of $G$ contains some snake.
\medskip

Given a set of edges $C$ in $\Gamma_G$ that constitute such a cycle,
assign a direction to each edge in $C$ so that the cycle becomes directed.
Each directed edge $i \rightarrow j$ in $C$ corresponds to some snake $S$ in $G$.
Assign the \emph{sign} $\sign(j-i)$ to every such snake
and all other snakes in $G$ are unsigned.

Since $C$ is a cycle, it follows that, for each row $R$ in $G$,
the number of entries in $R$ belonging to a positive snake,
is equal to the number of entries in $R$ that belongs to a negative snake.
This is called a \defin{signed snake cycle}.
Note that if $C$ is a \emph{simple} cycle, where $S_1$ and $S_2$ are two signed snakes in $C$
such that $S_1$ ends in the same row as $S_2$ starts, then $S_1$ and $S_2$ have the same sign.

A GT-pattern with a highlighted simple snake cycle is presented in Figure \ref{fig:signedSnakesExample}.
The corresponding graph $\Gamma_G$ is shown in Figure \ref{fig:snakeGraph}.
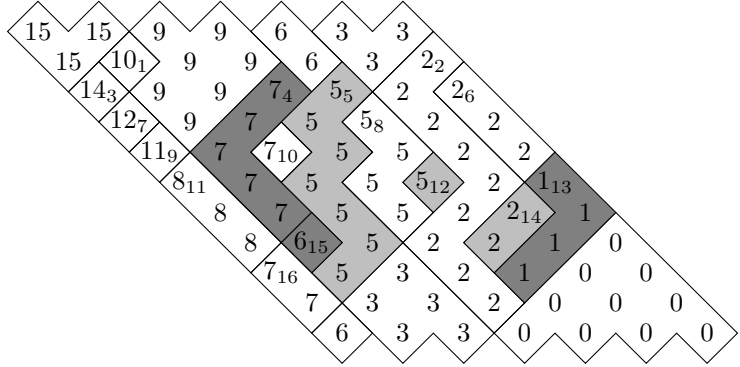
\begin{figure}[ht!]
\centering
\begin{tikzpicture}[scale=0.4]
\draw[color=black] (12,-10)--(13,-11)--(14,-10)--(13,-9)--(12,-10)--cycle;
\draw[color=black] (12,0)--(11,1)--(10,0)--(11,-1)--(12,-2)--(13,-1)--(12,0)--cycle;
\draw[color=black] (6,0)--(5,1)--(4,0)--(3,1)--(2,0)--(3,-1)--(4,-2)--(5,-1)--(6,0)--cycle;
\draw[color=black] (16,0)--(15,1)--(14,0)--(13,1)--(12,0)--(13,-1)--(14,-2)--(15,-1)--(16,0)--cycle;
\draw[color=black] (14,-8)--(13,-9)--(14,-10)--(15,-11)--(16,-10)--(17,-11)--(18,-10)--(17,-9)--(16,-8)--(15,-7)--(14,-8)--cycle;
\draw[color=black] (9,-3)--(10,-2)--(11,-1)--(10,0)--(9,1)--(8,0)--(7,1)--(6,0)--(7,-1)--(6,-2)--(7,-3)--(8,-4)--(9,-3)--cycle;
\draw[color=black] (21,-7)--(20,-8)--(19,-9)--(18,-10)--(19,-11)--(20,-10)--(21,-11)--(22,-10)--(23,-11)--(24,-10)--(25,-11)--(26,-10)--(25,-9)--(24,-8)--(23,-7)--(22,-6)--(21,-7)--cycle;
%
\draw[black] (5,-1)--(6,-2)--(7,-1)--(6,0)--(5,-1)--cycle;
\draw[black] (18,-4)--(17,-3)--(16,-2)--(17,-1)--(16,0)--(15,-1)--(14,-2)--(15,-3)--(16,-4)--(17,-5)--(16,-6)--(15,-7)--(16,-8)--(17,-9)--(18,-10)--(19,-9)--(18,-8)--(17,-7)--(18,-6)--(19,-5)--(18,-4)--cycle;
\draw[black] (4,-2)--(5,-3)--(6,-2)--(5,-1)--(4,-2)--cycle;
\filldraw[black,fill=gray] (11,-1)--(10,-2)--(9,-3)--(8,-4)--(9,-5)--(10,-6)--(11,-7)--(12,-6)--(11,-5)--(10,-4)--(11,-3)--(12,-2)--(11,-1)--cycle;
\filldraw[black,fill=lightgray] (11,-3)--(12,-4)--(11,-5)--(12,-6)--(13,-7)--(12,-8)--(13,-9)--(14,-8)--(15,-7)--(14,-6)--(13,-5)--(14,-4)--(13,-3)--(14,-2)--(13,-1)--(12,-2)--(11,-3)--cycle;
\draw[black] (20,-4)--(19,-3)--(18,-2)--(17,-1)--(16,-2)--(17,-3)--(18,-4)--(19,-5)--(20,-4)--cycle;
\draw[black] (5,-3)--(6,-4)--(7,-3)--(6,-2)--(5,-3)--cycle;
\draw[black] (15,-7)--(16,-6)--(15,-5)--(16,-4)--(15,-3)--(14,-2)--(13,-3)--(14,-4)--(13,-5)--(14,-6)--(15,-7)--cycle;
\draw[black] (6,-4)--(7,-5)--(8,-4)--(7,-3)--(6,-4)--cycle;
\draw[black] (10,-4)--(11,-5)--(12,-4)--(11,-3)--(10,-4)--cycle;
\draw[black] (11,-7)--(10,-6)--(9,-5)--(8,-4)--(7,-5)--(8,-6)--(9,-7)--(10,-8)--(11,-7)--cycle;
\filldraw[black,fill=lightgray] (15,-5)--(16,-6)--(17,-5)--(16,-4)--(15,-5)--cycle;
\filldraw[black,fill=gray] (20,-8)--(21,-7)--(22,-6)--(21,-5)--(20,-4)--(19,-5)--(20,-6)--(19,-7)--(18,-8)--(19,-9)--(20,-8)--cycle;
\filldraw[black,fill=lightgray] (17,-7)--(18,-8)--(19,-7)--(20,-6)--(19,-5)--(18,-6)--(17,-7)--cycle;
\filldraw[black,fill=gray] (11,-7)--(12,-8)--(13,-7)--(12,-6)--(11,-7)--cycle;
\draw[black] (12,-8)--(11,-7)--(10,-8)--(11,-9)--(12,-10)--(13,-9)--(12,-8)--cycle;
\node at (3,0) {$15$};
\node at (5,0) {$15$};
\node at (7,0) {$9$};
\node at (9,0) {$9$};
\node at (11,0) {$6$};
\node at (13,0) {$3$};
\node at (15,0) {$3$};
\node at (4,-1) {$15$};
\node at (6,-1) {$10_1$};
\node at (8,-1) {$9$};
\node at (10,-1) {$9$};
\node at (12,-1) {$6$};
\node at (14,-1) {$3$};
\node at (16,-1) {$2_2$};
\node at (5,-2) {$14_3$};
\node at (7,-2) {$9$};
\node at (9,-2) {$9$};
\node at (11,-2) {$7_4$};
\node at (13,-2) {$5_5$};
\node at (15,-2) {$2$};
\node at (17,-2) {$2_6$};
\node at (6,-3) {$12_7$};
\node at (8,-3) {$9$};
\node at (10,-3) {$7$};
\node at (12,-3) {$5$};
\node at (14,-3) {$5_8$};
\node at (16,-3) {$2$};
\node at (18,-3) {$2$};
\node at (7,-4) {$11_9$};
\node at (9,-4) {$7$};
\node at (11,-4) {$7_{10}$};
\node at (13,-4) {$5$};
\node at (15,-4) {$5$};
\node at (17,-4) {$2$};
\node at (19,-4) {$2$};
\node at (8,-5) {$8_{11}$};
\node at (10,-5) {$7$};
\node at (12,-5) {$5$};
\node at (14,-5) {$5$};
\node at (16,-5) {$5_{12}$};
\node at (18,-5) {$2$};
\node at (20,-5) {$1_{13}$};
\node at (9,-6) {$8$};
\node at (11,-6) {$7$};
\node at (13,-6) {$5$};
\node at (15,-6) {$5$};
\node at (17,-6) {$2$};
\node at (19,-6) {$2_{14}$};
\node at (21,-6) {$1$};
\node at (10,-7) {$8$};
\node at (12,-7) {$6_{15}$};
\node at (14,-7) {$5$};
\node at (16,-7) {$2$};
\node at (18,-7) {$2$};
\node at (20,-7) {$1$};
\node at (22,-7) {$0$};
\node at (11,-8) {$7_{16}$};
\node at (13,-8) {$5$};
\node at (15,-8) {$3$};
\node at (17,-8) {$2$};
\node at (19,-8) {$1$};
\node at (21,-8) {$0$};
\node at (23,-8) {$0$};
\node at (12,-9) {$7$};
\node at (14,-9) {$3$};
\node at (16,-9) {$3$};
\node at (18,-9) {$2$};
\node at (20,-9) {$0$};
\node at (22,-9) {$0$};
\node at (24,-9) {$0$};
\node at (13,-10) {$6$};
\node at (15,-10) {$3$};
\node at (17,-10) {$3$};
\node at (19,-10) {$0$};
\node at (21,-10) {$0$};
\node at (23,-10) {$0$};
\node at (25,-10) {$0$};
\end{tikzpicture}
\caption{
A GT-pattern $G$ with a cycle of snakes. Light gray indicate positive snakes and the dark gray snakes are negative.
The cycle corresponds to the snakes $4,15,14,12,13,5$ in that order.
}\label{fig:signedSnakesExample}
\end{figure}

\begin{figure}[!ht]
\centering
\begin{tikzpicture}[->,>=stealth',shorten >=1pt,auto,node distance=1cm,
  thick,main node/.style={circle,draw}]
\node[main node] (1) {1};
\node[main node] (2) [above of=1] {2};
\node[main node] (3) [above of=2] {3};
\node[main node] (4) [above of=3] {4};
\node[main node] (5) [above of=4] {5};
\node[main node] (6) [above of=5] {6};
\node[main node] (7) [above of=6] {7};
\node[main node] (8) [above of=7] {8};
\node[main node] (9) [above of=8] {9};
\node[main node] (10) [above of=9] {10};
\path[-,dashed,every node/.style={font=\sffamily\small}]
   (1) edge [bend right=80,looseness=1] node [right] {2} (10)
   (9) edge [bend right=20] node [right] {1} (10)
  (8) edge [bend right=20] node [right] {3} (9)
  (4) edge [solid,<-,bend right=70] node [right] {4} (9)
  (2) edge [solid,->,bend left=100,looseness=1.5] node [left] {5} (9)
 (6) edge [bend left=60] node [left] {6} (9)
(7) edge [bend left=20] node [left] {7} (8)
  (4) edge [bend right=60] node [right] {8} (8)
(6) edge [bend left=20] node [left] {9} (7)
(6) edge [bend right=20] node [right] {10} (7)
(3) edge [bend left=90,looseness=1.7] node [left] {11} (6)
(5) edge [solid,->,bend right=20] node [right] {12} (6)
(2) edge [solid,<-,bend left=90,looseness=1.8] node [left] {13} (6)
(3) edge [solid,->,bend left=70] node [left] {14} (5)
(3) edge [solid,<-,bend right=20] node [right] {15} (4)
(1) edge [bend right=70] node [right] {16} (3);
\end{tikzpicture}
\caption{
The graph $\Gamma_G$ corresponding to the GT-pattern in
\ref{fig:signedSnakesExample}. The solid directed edges corresponds to the snake cycle,
and edge labels corresponds to the corresponding snakes in $G$.
Note that there are several choices of a simple snake cycle in $G$.
}\label{fig:snakeGraph}
\end{figure}
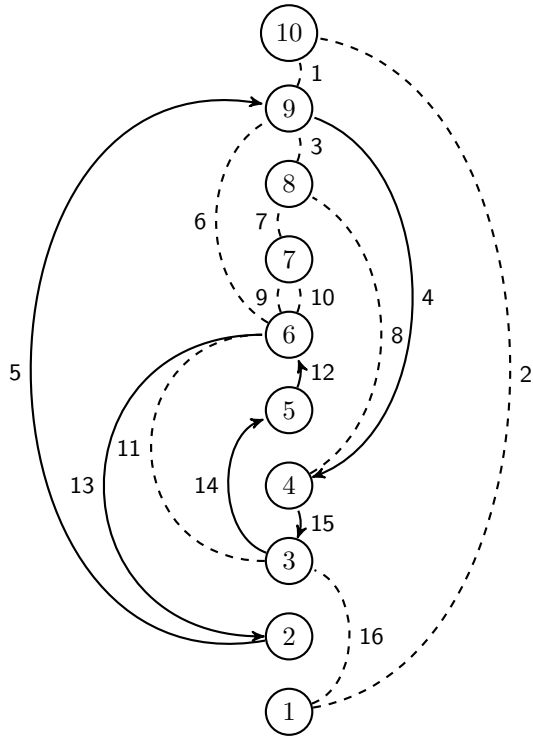

\bigskip

Let $G \in \GT_{k\lambdavec/k\muvec,k\nuvec}$ and let $C$ be a set of signed snakes in the canonical snake partition of $G$.
By repeatedly applying Lemma \ref{lem:snakemoves} on the signed snakes in a tile $T$,
one can reach a configuration where all positive snakes in $T$ are to the left of all non-signed snakes,
and all negative snakes in $T$ are to the right of the non-signed snakes in $T$.
When this is done in every tile of $G$, we have obtained a new proper snake partition of $G$,
where we say that the snakes in $G$ have been sign-sorted with respect to $C$.

Consider a sign-sorted proper snake partition of a GT-pattern $G$ obtained from a snake cycle $C$.
Let $c$ be a natural number, and consider the set of snakes in $G$ with
\begin{itemize}
\item content $c$ and positive sign and
\item content $c+1$ and negative sign.
\end{itemize}
Consider the connected components of this set. 
One can prove that each such connected component has the shape of a free tile,
using the same reasoning as in Lemma \ref{lem:tileunion}.
Call such a component a \defin{mixed} tile. 
Notice that each mixed tile has a snake partition consisting of snakes of one or two of the types listed above.

\begin{lemma}\label{lem:mixedsort}
Let $G$ be GT-pattern with a sign-sorted snake partition obtained from a snake cycle $C$ and let $T$ be a mixed tile in $G$.
If $C$ is a \emph{simple} cycle, then the snake partition of $T$ is proper.
\end{lemma}
\begin{proof}
If all snakes in $T$ have the same sign, then these snakes is a subset of a proper 
snake partition of a tile in $G$ and the statement follows.

Thus, it is enough to show that if $S_-$ and $S_+$ are two snakes with opposite signs in $T$, 
then 
\begin{itemize}
\item $S_-$ and $S_+$ do not start or end in the same row, and 
\item $S_-$ do not start (end) in the same row as $S_+$ end (start).
\end{itemize}

Note that $S_+$ is a member of some tile with content $c$ and that $S_-$ is a member of a tile with content $c+1$.
Thus, according to Lemma \ref{lem:adjacenttiles}, the first situation above is not possible.

The second situation corresponds to two edges in the snake cycle $C$ that 
both enters or exits the same vertex in $\Gamma_G$. 
This is not possible if $C$ is simple.
\end{proof}

\begin{proposition}\label{prop:cycleaction}
Let $G \in \GT_{k\lambdavec/k\muvec,k\nuvec}$ and let $C$ be a simple cycle of
signed snakes in the canonical snake partition of $G$.
Adding $1$ to all entries in positive snakes and $-1$ to all entries in negative snakes,
followed by reordering the entries in each row decreasingly, gives a GT-pattern $G'$
in $\GT_{k\lambdavec/k\muvec,k\nuvec}$.
\end{proposition}
\begin{proof}
Start with sign-sorting the snakes in $G$ with respect to $C$.
If there are no mixed tiles in $G$ that contains snakes with two different signs,
we can repeatedly use Lemma \ref{lem:addingisok} to add $1$ to all entries of $G$ in positive snakes
and subtract $1$ from all negative snakes and obtain a GT-pattern $G'$.
The sign-sorting of the signed snakes corresponds precisely to sorting the rows in decreasing order.

However, if $G$ do contain mixed tiles with snakes of both signs,
then the above procedure might fail to produce a proper GT-pattern $G'$;
a negative snake with content $c+1$ to the left of a positive snake with content $c$,
will not yield a proper GT-pattern when adding $1$ to the positive 
snake and subtracting $1$ from the negative snake.

The solution to this is as follows: Lemma \ref{lem:mixedsort} ensures 
that the snake partition of a mixed tile is proper.
Lemma \ref{lem:snakemoves} can therefore be applied repeatedly in every mixed tile, 
such that the positive snakes in these tiles are placed to the left of all negative snakes.
After this sorting procedure, it is safe to add and subtract $1$ as above.

Lemma \ref{lem:addingisok} ensures that all inequalities that 
needs to be fulfilled in a GT-pattern hold in the result $G'$

\end{proof}

Figure \ref{fig:signedSnakesExample2} illustrate Proposition \ref{prop:cycleaction}.
Note that the snake cycle in Proposition \ref{prop:cycleaction} must be simple.
If it is not, the result might not be a valid GT-pattern, as shown in Figure \ref{fig:cycleproblem}.
\begin{figure}[ht!]
\centering
\begin{tikzpicture}[scale=0.4]
\filldraw[color=black,fill=white] (12,-10)--(13,-11)--(14,-10)--(13,-9)--(12,-10)--cycle;
\filldraw[color=black,fill=white] (6,0)--(5,1)--(4,0)--(3,1)--(2,0)--(3,-1)--(4,-2)--(5,-1)--(6,0)--cycle;
\filldraw[color=black,fill=white] (16,0)--(15,1)--(14,0)--(13,1)--(12,0)--(13,-1)--(14,-2)--(15,-1)--(16,0)--cycle;
\filldraw[color=black,fill=white] (9,-3)--(10,-2)--(11,-1)--(10,0)--(9,1)--(8,0)--(7,1)--(6,0)--(7,-1)--(6,-2)--(7,-3)--(8,-4)--(9,-3)--cycle;
\filldraw[color=black,fill=white] (17,-5)--(16,-6)--(15,-7)--(14,-8)--(13,-9)--(14,-10)--(15,-11)--(16,-10)--(17,-11)--(18,-10)--(17,-9)--(16,-8)--(17,-7)--(18,-6)--(17,-5)--cycle;
\filldraw[color=black,fill=white] (19,-5)--(20,-6)--(19,-7)--(18,-8)--(19,-9)--(18,-10)--(19,-11)--(20,-10)--(21,-11)--(22,-10)--(23,-11)--(24,-10)--(25,-11)--(26,-10)--(25,-9)--(24,-8)--(23,-7)--(22,-6)--(21,-5)--(20,-4)--(19,-5)--cycle;
\filldraw[color=black,fill=white] (14,-8)--(15,-7)--(14,-6)--(15,-5)--(14,-4)--(13,-3)--(14,-2)--(13,-1)--(12,0)--(11,1)--(10,0)--(11,-1)--(10,-2)--(9,-3)--(10,-4)--(9,-5)--(10,-6)--(11,-7)--(12,-8)--(13,-9)--(14,-8)--cycle;
%
\filldraw[black,dashed,fill=gray] (11,-1)--(10,-2)--(9,-3)--(10,-4)--(9,-5)--(10,-6)--(11,-7)--(12,-6)--(11,-5)--(12,-4)--(11,-3)--(12,-2)--(11,-1)--cycle;
%
\filldraw[black,dashed,fill=lightgray] (11,-3)--(12,-4)--(11,-5)--(12,-6)--(11,-7)--(12,-8)--(13,-9)--(14,-8)--(13,-7)--(14,-6)--(13,-5)--(14,-4)--(13,-3)--(14,-2)--(13,-1)--(12,-2)--(11,-3)--cycle;
\filldraw[black,dashed,fill=lightgray] (13,-5)--(14,-6)--(15,-5)--(14,-4)--(13,-5)--cycle;
%
\filldraw[black,dashed,fill=lightgray] (15,-7)--(16,-8)--(17,-7)--(18,-6)--(17,-5)--(16,-6)--(15,-7)--cycle;
\filldraw[black,dashed,fill=gray] (20,-8)--(21,-7)--(22,-6)--(21,-5)--(20,-4)--(19,-5)--(20,-6)--(19,-7)--(18,-8)--(19,-9)--(20,-8)--cycle;
\filldraw[black,dashed,fill=gray] (13,-7)--(14,-8)--(15,-7)--(14,-6)--(13,-7)--cycle;

\draw[black] (5,-1)--(6,-2)--(7,-1)--(6,0)--(5,-1)--cycle;
\draw[black] (18,-8)--(19,-7)--(20,-6)--(19,-5)--(18,-4)--(17,-3)--(16,-2)--(17,-1)--(16,0)--(15,-1)--(14,-2)--(15,-3)--(16,-4)--(17,-5)--(18,-6)--(17,-7)--(16,-8)--(17,-9)--(18,-10)--(19,-9)--(18,-8)--cycle;
\draw[black] (4,-2)--(5,-3)--(6,-2)--(5,-1)--(4,-2)--cycle;
\draw[black] (20,-4)--(19,-3)--(18,-2)--(17,-1)--(16,-2)--(17,-3)--(18,-4)--(19,-5)--(20,-4)--cycle;

\draw[black] (5,-3)--(6,-4)--(7,-3)--(6,-2)--(5,-3)--cycle;
%
\draw[black] (6,-4)--(7,-5)--(8,-4)--(7,-3)--(6,-4)--cycle;
\draw[black] (8,-4)--(9,-5)--(10,-4)--(9,-3)--(8,-4)--cycle;
\draw[black] (11,-7)--(10,-6)--(9,-5)--(8,-4)--(7,-5)--(8,-6)--(9,-7)--(10,-8)--(11,-7)--cycle;
 \draw[black] (12,-8)--(11,-7)--(10,-8)--(11,-9)--(12,-10)--(13,-9)--(12,-8)--cycle;
\node at (3,0) {$15$};
\node at (5,0) {$15$};
\node at (7,0) {$9$};
\node at (9,0) {$9$};
\node at (11,0) {$6$};
\node at (13,0) {$3$};
\node at (15,0) {$3$};
\node at (4,-1) {$15$};
\node at (6,-1) {$10_1$};
\node at (8,-1) {$9$};
\node at (10,-1) {$9$};
\node at (12,-1) {$6$};
\node at (14,-1) {$3$};
\node at (16,-1) {$2_2$};
\node at (5,-2) {$14_3$};
\node at (7,-2) {$9$};
\node at (9,-2) {$9$};
\node at (11,-2) {$6_4$};
\node at (13,-2) {$6_5$};
\node at (15,-2) {$2$};
\node at (17,-2) {$2_6$};
\node at (6,-3) {$12_7$};
\node at (8,-3) {$9$};
\node at (10,-3) {$6$};
\node at (12,-3) {$6$};
\node at (14,-3) {$5_8$};
\node at (16,-3) {$2$};
\node at (18,-3) {$2$};
\node at (7,-4) {$11_9$};
\node at (9,-4) {$7_{10}$};
\node at (11,-4) {$6$};
\node at (13,-4) {$6$};
\node at (15,-4) {$5$};
\node at (17,-4) {$2$};
\node at (19,-4) {$2$};
\node at (8,-5) {$8_{11}$};
\node at (10,-5) {$6$};
\node at (12,-5) {$6$};
\node at (14,-5) {$6_{12}$};
\node at (16,-5) {$5$};
\node at (18,-5) {$2$};
\node at (20,-5) {$0_{13}$};
\node at (9,-6) {$8$};
\node at (11,-6) {$6$};
\node at (13,-6) {$6$};
\node at (15,-6) {$5$};
\node at (17,-6) {$3_{14}$};
\node at (19,-6) {$2$};
\node at (21,-6) {$0$};
\node at (10,-7) {$8$};
\node at (12,-7) {$6$};
\node at (14,-7) {$5_{15}$};
\node at (16,-7) {$3$};
\node at (18,-7) {$2$};
\node at (20,-7) {$0$};
\node at (22,-7) {$0$};
\node at (11,-8) {$7_{16}$};
\node at (13,-8) {$6$};
\node at (15,-8) {$3$};
\node at (17,-8) {$2$};
\node at (19,-8) {$0$};
\node at (21,-8) {$0$};
\node at (23,-8) {$0$};
\node at (12,-9) {$7$};
\node at (14,-9) {$3$};
\node at (16,-9) {$3$};
\node at (18,-9) {$2$};
\node at (20,-9) {$0$};
\node at (22,-9) {$0$};
\node at (24,-9) {$0$};
\node at (13,-10) {$6$};
\node at (15,-10) {$3$};
\node at (17,-10) {$3$};
\node at (19,-10) {$0$};
\node at (21,-10) {$0$};
\node at (23,-10) {$0$};
\node at (25,-10) {$0$};
\end{tikzpicture}
\caption{The result after completing the snake moves in
the highlighted snake cycle in in Figure \ref{fig:signedSnakesExample}.
Entries are shaded to illustrate the final positions of the snakes after sorting.
}\label{fig:signedSnakesExample2}
\end{figure}

\begin{figure}[ht!]
\centering
\begin{tikzpicture}[scale=0.4]
\draw[black] (6,-2)--(5,-1)--(4,0)--(3,1)--(2,0)--(3,-1)--(4,-2)--(5,-3)--(6,-2)--cycle;
\draw[black] (4,0)--(5,-1)--(6,0)--(5,1)--(4,0)--cycle;
\draw[black] (6,0)--(7,-1)--(8,0)--(7,1)--(6,0)--cycle;
\filldraw[black,fill=lightgray] (5,-1)--(6,-2)--(7,-1)--(6,0)--(5,-1)--cycle;
\filldraw[black,fill=gray] (7,-1)--(8,-2)--(9,-1)--(8,0)--(7,-1)--cycle;
\filldraw[black,fill=gray] (6,-2)--(7,-3)--(8,-2)--(7,-1)--(6,-2)--cycle;
\filldraw[black,fill=lightgray] (8,-2)--(9,-3)--(10,-2)--(9,-1)--(8,-2)--cycle;
\draw[black] (5,-3)--(6,-4)--(7,-3)--(6,-2)--(5,-3)--cycle;
\draw[black] (7,-3)--(8,-4)--(9,-3)--(8,-2)--(7,-3)--cycle;
\draw[black] (9,-3)--(10,-4)--(11,-3)--(10,-2)--(9,-3)--cycle;
\node at (3,0) {$12$};
\node at (5,0) {$12$};
\node at (7,0) {$9$};
\node at (4,-1) {$12$};
\node at (6,-1) {$10_1$};
\node at (8,-1) {$5_2$};
\node at (5,-2) {$12$};
\node at (7,-2) {$8_3$};
\node at (9,-2) {$4_4$};
\node at (6,-3) {$9$};
\node at (8,-3) {$6$};
\node at (10,-3) {$3$};
\end{tikzpicture}
\begin{tikzpicture}[scale=0.4]
\draw[black] (6,-2)--(5,-1)--(4,0)--(3,1)--(2,0)--(3,-1)--(4,-2)--(5,-3)--(6,-2)--cycle;
\draw[black] (4,0)--(5,-1)--(6,0)--(5,1)--(4,0)--cycle;
\draw[black] (6,0)--(7,-1)--(8,0)--(7,1)--(6,0)--cycle;
\draw[black] (5,-1)--(6,-2)--(7,-1)--(6,0)--(5,-1)--cycle;
\draw[black] (7,-1)--(8,-2)--(9,-1)--(8,0)--(7,-1)--cycle;
\draw[black] (6,-2)--(7,-3)--(8,-2)--(7,-1)--(6,-2)--cycle;
\draw[black] (8,-2)--(9,-3)--(10,-2)--(9,-1)--(8,-2)--cycle;
\draw[black] (5,-3)--(6,-4)--(7,-3)--(6,-2)--(5,-3)--cycle;
\draw[black] (7,-3)--(8,-4)--(9,-3)--(8,-2)--(7,-3)--cycle;
\draw[black] (9,-3)--(10,-4)--(11,-3)--(10,-2)--(9,-3)--cycle;
\node at (3,0) {$12$};
\node at (5,0) {$12$};
\node at (7,0) {$9$};
\node at (4,-1) {$12$};
\node at (6,-1) {$11$};
\node at (8,-1) {$\mathbf 4$};
\node at (5,-2) {$12$};
\node at (7,-2) {$7$};
\node at (9,-2) {$\mathbf 5$};
\node at (6,-3) {$9$};
\node at (8,-3) {$6$};
\node at (10,-3) {$3$};
\end{tikzpicture}
\caption{The condition that $C$ is a simple cycle in Proposition \ref{prop:cycleaction} is important.
The non-simple snake cycle in the left GT-pattern above does not give a proper GT-pattern.
The reason is that the snakes $2$ and $4$ in the first pattern
has opposite signs, even though snake $2$ ends where snake $4$ starts.
}\label{fig:cycleproblem}
\end{figure}

Given a GT-pattern $G$, consider the vector
$\vvec_G=(x^1_1,x^1_2,\dots,x^1_n,x^2_1,\dots,x^m_n)$,
given by the concatenation of the rows $1,2,\dots,m$ in $G$.
For two GT-patterns $G$ and $G'$ of the same size,
we write $G \lexless G'$ if $\vvec_{G}$ is (strictly) lexicographically smaller than $\vvec_{G'}$.

\begin{lemma}\label{lem:ordering}
Let $G \in \GT_{k\lambdavec/k\muvec,k\nuvec}$ where $C_1$ is a simple signed snake cycle in $G$.
Let $C_2$ be the same cycle as $C_1$ but with the opposite orientation of the edges.
Using Proposition \ref{prop:cycleaction}, $C_1$ and $C_2$ give rise to two new GT-patterns $G_1$ and $G_2$, respectively.

Then $G_1 \lexless G \lexless G_2$ or $G_2 \lexless G \lexless G_1$.
\end{lemma}
\begin{proof}
Consider the first row $i$ in $G$ which intersects some snakes in $C_1$.
This row must contain then contain the lowest entry of exactly two snakes $S_1$ and $S_2$
in $C_1$, since $C_1$ is a simple cycle. We may assume that
the lowest entry in $S_1$ is to the left of the lowest entry in $S_2$.
The snakes $S_1$ and $S_2$ are members of two different free tiles, $T_1$, $T_2$,
and the content of these tiles must differ by at least two, according to
Lemma \ref{lem:snakesstartstop} and Lemma \ref{lem:adjacenttiles}.

Thus, using Proposition \ref{prop:cycleaction} to obtain $G_1$ or $G_2$
either increases the leftmost entry, or decreases the rightmost entry of row $i$ in $T_1$.
Note that both these entries appear (in $\vvec_G$) before any other entry affected by this action.

Thus, if $C_1$ makes $S_1$ negative, it is evident that $G_1 \lexless G \lexless G_2$. 
In the case of a positive $S_1$, we have that $G_2 \lexless G \lexless G_1$.
\end{proof}
For example, the snake cycle in $G$ in Figure \ref{fig:signedSnakesExample}
gives a GT-pattern $G'$ such that $G \lexless G'$,
the first entry which differs between $\vvec_G$ and $\vvec_{G'}$
corresponds to the lowest entry in snake $5$.
\medskip 

\begin{corollary}\label{cor:largestismultiple}
If $\GT_{k\lambdavec/k\muvec,k\nuvec}$ is non-empty, then the lexicographically 
largest (smallest) element $G$ in $\GT_{k\lambdavec/k\muvec,k\nuvec}$
has no snakes in the snake partition. That is, all entries in $G$ are multiples of $k$.
\end{corollary}
%

\medskip

Theorem \ref{thm:mainthm} now follows from Corollary \ref{cor:largestismultiple}:
\begin{proof}[Proof of Theorem \ref{thm:mainthm}]
The $(\Leftarrow)$ part of Theorem \ref{thm:mainthm} has already been addressed, so it suffices to prove 
\[
|\GT_{k\lambdavec/k\muvec,k\nuvec}>0| \Rightarrow |\GT_{\lambdavec/\muvec,\nuvec}| >0.
\]
Corollary \ref{cor:largestismultiple} implies that all entries in the lex-largest element $G$ in 
$\GT_{k\lambdavec/k\muvec,k\nuvec}$ are multiples of $k$.
Therefore, $\frac{1}{k}G \in \GT_{\lambdavec/\muvec,\nuvec}$,
so this set is non-empty.
\end{proof}

\medskip

The same technique can be used to prove a generalization of $K$-Fultons Theorem.
This also follows from a result in \cite{Knutson00}, where the corresponding statement was proved for the Littlewood-Richardson coefficients.
and as mentioned before, skew Kostka numbers are special cases of Littlewood-Richardson coefficients.
\begin{proposition}[$K$-Fultons Theorem]\label{eqn:kequalsone}
For fixed partitions $\lambdavec,\muvec$ and $\nuvec,$ the following holds for any integer $k \geq 1:$
\begin{align*}
K_{k\lambdavec/k\muvec,k\nuvec} = 1 \Leftrightarrow K_{\lambdavec/\muvec,\nuvec} = 1.
\end{align*}
\end{proposition}
\begin{proof}
Using Theorem \ref{thm:mainthm}, it suffices to show 
that $|\GT_{k\lambdavec/k\muvec,k\nuvec}|>1 \Leftrightarrow |\GT_{\lambdavec/\muvec,\nuvec}|>1$.

The $(\Leftarrow)$ direction is trivial, every element in $\GT_{\lambdavec/\muvec,\nuvec}$
can (injectively) be mapped to an element $\GT_{k\lambdavec/k\muvec,k\nuvec}$ by entrywise multiplication by $k$.

Assume $G_1,G_2 \in \GT_{k\lambdavec/k\muvec,k\nuvec}$, where $G_1 \lexless G_2$.
Corollary \ref{cor:largestismultiple} then implies 
that there are $G_{min},G_{max} \in \GT_{k\lambdavec/k\muvec,k\nuvec}$
such that $G_{min} \lexleq G_1 \lexless G_2 \lexleq G_{max}$
and all entries in $G_{min}$ and $G_{max}$ are multiples of $k$.

Hence, $\frac{1}{k}G_{min}$ and $\frac{1}{k}G_{max}$ are two different elements in
$\GT_{\lambdavec/\muvec,\nuvec}$, which concludes the proof.
\end{proof}


\section{Final remarks and refinements on open questions}

Given $G \in GT_{\lambdavec/\muvec,\nuvec}$, consider the function 
\[
 p_G(k) = |\{ G' \in GT_{k\lambdavec/k\muvec,k\nuvec} | G' \lexleq k G \}|.
\]
Corollary \ref{cor:largestismultiple} implies that for the lex-largest element $G_{max}$ in  $GT_{\lambdavec/\muvec,\nuvec}$,
the value of $p_{G_{max}}(k)$ is just the skew Kostka number $K_{k\lambdavec/k\muvec,k\nuvec}$.
It is known \cite{Kirillov88thebethe,Rassart2004} that
$k \mapsto K_{k\lambdavec/k\muvec,k\nuvec}$ is a polynomial in $k$,
which motivates the following conjecture:
\begin{conjecture}\label{conj:polynomiality}
For any fixed $G \in GT_{\lambdavec/\muvec,\nuvec}$, the function $p_G(k)$ is polynomial in $k$.
\end{conjecture}

The function $k \mapsto K_{k\lambdavec/k\muvec,k\nuvec}$ can be interpreted
as the Ehrhart function associated with certain \emph{rational} polytopes.
Not much except the degrees (in the non-skew case) of these polynomials are known, see \cite{Mcallister2008}.
It is immediate from Ehrhart theory that the functions $p_G(k)$ in Conjecture \ref{conj:polynomiality}
are quasi-polynomial, since $p_G(k)$ can also be seen as counting lattice points inside some
rational polytope.
However, the study of non-integral, rational polytopes having polynomial Erhart functions is
still very much an open research topic, see \cite{McAllister2005}, and few examples of such polytopes
are known.
Thus, Conjecture \ref{conj:polynomiality} give a new family of such polytopes.


\medskip 

In \cite{King04stretched}, it is conjectured that the polynomial $k \mapsto K_{k\lambdavec,k\nuvec}$ 
has only non-negative coefficients. The same seem to hold for the $p_G(k)$:
\begin{conjecture}\label{conj:nonnegativity}
For any fixed $G \in GT_{\lambdavec/\muvec,\nuvec}$ the function $p_G(k)$ is a polynomial in $k$,
with non-negative coefficients.
\end{conjecture}
This is supported by extensive computer experiments, where over 20000 choices
of (small) parameters of $\lambdavec/\muvec$, $\nuvec$ have been verified to support this conjecture.

\begin{remark}
Note that the polynomial $k \mapsto K_{k\lambdavec/k\muvec,k\nuvec}$
does not depend on the order of the entries in $\nuvec$ (it does not need to be a partition),
but the functions obtained as $p_G(k)$ \emph{do} depend 
on the order of the entries in $\nuvec$ in general.
\end{remark}

\medskip
Finally, it is worth mentioning a related conjecture from \cite{King04stretched},
stating that for each natural number $m$, 
there is only a finite number of polynomials 
of the form $p(k)=K_{k\lambdavec,k\nuvec}$ such that $p(1)=m$.
The cases $m=0$ and $m=1$ are established by the Saturation Theorem, and K-Fultons Theorem
respectively. 


\bibliographystyle{amsplain}
\bibliography{bibliography}

\end{document}